\documentclass[a4paper,12pt]{article}
\UseRawInputEncoding
\usepackage{a4wide}
\usepackage[utf8]{inputenc}
\usepackage{amsmath,amssymb,amsthm, amsfonts,color}
\usepackage{amsmath,amsthm,amssymb,amsfonts}
\usepackage[colorlinks=true,citecolor=black,linkcolor=black]{hyperref}
\usepackage{ upgreek }
\usepackage{authblk}
\usepackage{bbm}

\usepackage[dvips]{graphicx}

\usepackage[dvips]{graphicx}
\usepackage{psfrag}
\usepackage{xcolor}
\usepackage{tikz}
\usetikzlibrary{shapes,arrows,positioning}
\usepackage{hyperref}
\usepackage{constants}

\newtheorem{theorem}{Theorem}[section]
\newtheorem{proposition}[theorem]{Proposition}
\newtheorem{lemma}[theorem]{Lemma}

\usepackage[utf8]{inputenc}
\newtheorem{remark}[theorem]{Remark}

\newtheorem{thmy}{Theorem}

\newcommand{\vbgs}[4]{
\xymatrix{#1 \ar[d]_{\tilde p} \ar@<2pt>[r] \ar@<-2pt>[r] & #2 \ar[d]^p \\ #3 \ar@<2pt>[r] \ar@<-2pt>[r] & #4}}

\begin{document}
\newcommand{\vbal}[4]{
\xymatrix{#1 \ar[d]_{\tilde p} \ar[r] & #2 \ar[d]^p \\ #3 \ar[r]  & #4}}
\newcommand\lie{\mathfrak}
\newenvironment{dedication}
  {
   \thispagestyle{empty}
   \itshape             
   \raggedleft          
  }
  {\par 
  }

\title{\textbf{Geometric study of non-constant vector fields making hyperbolic space Ricci-Bourguignon solitons}}
\author{ Mafal Ndiaye Diop\thanks{ Address:
    D\'epartement de Math\'ematiques et Informatique, FST, Universit\'e Cheikh Anta Diop, Dakar,  S\'en\'egal, Email: mafaldiop95@gmail.com},
   Abdou Bousso\thanks{D\'epartement de Math\'ematiques et Informatique, FST, Universit\'e Cheikh Anta Diop, Dakar,  S\'en\'egal,      Email:abdoukskbousso@gmail.com}, 
   Cheikh Khoul\'e \thanks{Address:CERER, Universit\'e Cheikh Anta Diop, and  Dakar,  S\'en\'egal,
    Email:cheikh1.khoule@ucad.edu.sn}, \\ and 
    Ameth Ndiaye \thanks{Address:D\'epartement de Mathematiques, FASTEF, Universit\'e Cheikh Anta Diop,  Dakar,  S\'en\'egal,
    Email:ameth1.ndiaye@ucad.edu.sn} }

\date{}
\maketitle

\begin{abstract}
The objective of this paper is to deepen the study of vector fields on hyperbolic spaces $\mathbb{H}^n$ that transform them into a Ricci-Bourguignon soliton. Starting from a recent work in \cite{bousso2025ricci} which characterizes these fields as Killing fields of a specific shape, we propose a detailed geometric study of their structure and behavior in the dimensions $n=2, 3$ and $n\geq 3$. In odd dimension we show that the dual form of these vectors are contact forms. This work aims to enrich the understanding of self-similar solutions of the Ricci flow in a more general context.
\end{abstract}
    


\tableofcontents

\section{Introduction}
The Ricci flow, introduced by Hamilton \cite{Hamilton}, is a fundamental tool in differential geometry for studying the deformation of Riemannian metrics on a manifold. A particularly important class of solutions to this flow are Ricci solitons, which represent self-similar solutions. A vector field, often referred to as a soliton vector field, plays a crucial role in the definition of these structures, which can be seen as critical points for the Ricci flow on a manifold. Later, in their research, Catino et al. \cite{Bourguignon} developed the Ricci-Bourguignon flow into parabolic theory, which Bourguignon \cite{Catino} had first studied. \\
In this field, hyperbolic spaces \cite{fasihi2025hyperbolic}, of which the Poincar\'e half-plane is an example in dimension $n=2$, are geometrical spaces of great interest. Their non-Euclidean geometry provides a rich framework for the exploration of various geometrical equations. The study of Ricci solitons on these spaces has been the subject of several recent works, which have shed light on the nature of these solutions in a non-compact setting. In their paper, Fasihi-Ramandi et al. \cite{fasihi2025hyperbolic} established a remarkable result by showing that the set of vector fields that transform the Poincar\'e semiplane into a Ricci soliton is isomorphic to the Lie algebra $\text{sl}_2(\mathbb{R})$. Relatedly, the first and the third authors \cite{bousso2025ricci} have recently shown that the vector fields that make $\mathbb{H}^n$ a Ricci-Bourguignon soliton are Killing fields and have a precise algebraic form. \\
The rest of this article is organized as follows. In the \ref{sec:rappel} section, we recall the key definitions and equations of classical Ricci and Ricci-Bourguignon solitons. We also present the results of Bousso and Ndiaye. Sections \ref{sec:n2}, \ref{sec:n3} and \ref{sec:n4} are devoted to the study of soliton vector fields on hyperbolic spaces in dimension $n=2$, $n=3$ and $n\geq 3$, respectively.

\section{Preliminaries}\label{sec:rappel}
We begin by recalling the definition of a Ricci-Bourguignon soliton. Bourguignon introduced the Ricci-Bourguignon flow given by the equation
\begin{eqnarray}\label{BF}
\frac{\partial g}{\partial t} = -2(S - \rho r g), \quad g(0) = g_0
\end{eqnarray}
where $\rho$ is a non-zero real number, $S$ is Ricci tensor and $r$ is the scalar curvature. Ricci Bourguignon soliton (abbrivated as RB-soliton) is self similar solution of RB-flow and is defined on pseudo(semi)-Riemannian or Riemannaian manifold as follows:
\begin{equation} \label{eq:1.2}
(\mathcal{L}_V g)(X, Y ) = -2S(X, Y ) + 2(\lambda + \rho r)g(X, Y ),
\end{equation}
where $\mathcal{L}_V$ is Lie-derivative along $V$, $\lambda$ is soliton constant. If $\rho = 0$ then \eqref{eq:1.2} defines Ricci soliton. If the vector field $V$ is gradient of potential function $u$ then $g$ is said to be \textbf{gradient Ricci--Bourguignon soliton} and the equation (1.2) takes the form
\begin{equation}
\label{eq:1.3}
\text{Hess}u = -S(X, Y ) + (\lambda + \rho r)g(X, Y ),
\end{equation}
where $\text{Hess}u$ denotes the Hessian of smooth function $u$ on $M$ and is defined by $\text{Hess}u = \nabla\nabla u$. Recently many authors studied RB-soliton on three-dimensional contact metric manifold (\cite{3, 4, 8}). Pathra, Ali and Mofarreh[8] studied the RB-soliton and gradient RB-soliton on $k$-paracontact, $(k, \mu)$-paracontact and para-Sasakian manifolds.

A vector field $V$ on a pseudo-Riemannian manifold $M$ is \textbf{affine conformal}, \textbf{conformal}, \textbf{Jacobi type} and \textbf{torse-forming} vector field if([11], [18], [6])
\begin{align}
\label{eq:1.4}
(\mathcal{L}_V \nabla)(X, Y ) &= (Xf)Y + (Y f)X - g(X, Y )Df, \\
\label{eq:1.5}
(\mathcal{L}_V g)(X, Y ) &= 2\delta g(X, Y ), \\
\label{eq:1.6}
\nabla_X\nabla_Y V - \nabla_{\nabla_X Y} V + R(V, X)Y &= 0, \\
\label{eq:1.7}
\nabla_X V &= fX + \alpha(X)V,
\end{align}
where $X$, $Y$ and $Z$ are smooth vector fields on $M$, $\nabla$ is Levi-Civita connection on $M$, $f$ and $\delta$ are smooth functions on $M$ and $\alpha$ is a 1-form on $M$.

If $\alpha = 0$ in (\ref{eq:1.7}) then $V$ is \textbf{concircular}, $\alpha = 0$ and $f = 1$ in (\ref{eq:1.7}) then $V$ is a \textbf{concurrent vetor field}, and if $\alpha \neq 0$ and $f = 0$ in (\ref{eq:1.7}) then $V$ is \textbf{recurrent}.

If a $(0, 2)$-tensor $B$ satisfies the following
\begin{equation}
\label{eq:1.8}
(\nabla_X B)(Y, Z) + (\nabla_Y B)(Z, X) + (\nabla_Z B)(X, Y ) = k(X)g(Y, Z) + k(Y )g(Z, X) + k(Z)g(X, Y )
\end{equation}
for any vector fields $X$, $Y$ and $Z$ on $M$, then it is known as the \textbf{conformal quadratic Killing tensor}, which is a generalization of conformal Killing vector field, where $k$ is a 1-form on $M$.

A recent result by Bousso and Ndiaye \cite{bousso2025ricci} shows that the vector field $X$ that make the hyperbolic space $\mathbb{H}^n$ a Ricci-Bourguignon soliton is Killing vector field and is of the form:
\begin{align}\label{X}
X(x_1,\dots,x_n)&=\sum_{k=1}^{n-1}\left(\frac{a_k}{2}\left(x_k^2-\sum_{j\neq k}x_j^2\right)
+\left(\sum_{\substack{i=1\\i\neq k}}^{n-1}a_i\,x_i+b\right)x_k+c_k\right)\frac{\partial}{\partial x_k}\nonumber\\
&\quad+\left(\left(\sum_{k=1}^{n-1}a_k\,x_k+b\right)x_n\right)\frac{\partial}{\partial x_n}.
\end{align}
In this paper we denote by \(\Gamma_n\) the set of non-contant vector fields making $\mathbb{H}^n$  Ricci-Bourguignon soliton.
 
\subsection{Lie algebra ans Lie subalgebra}

A \textbf{Lie algebra} $\mathfrak{g}$ is a vector space on a field $\mathbb{F}$ (here, $\mathbb{R}$) with a bilinear map, denoted $[\cdot, \cdot]$ and called \textbf{Lie bracket}, which satisfies the following properties for any $X, Y, Z \in \mathfrak{g}$ and $\alpha, \beta \in \mathbb{F}$:
\begin{enumerate}
    \item \textbf{Bilinearity}: The bracket is linear in each argument.
    $$[\alpha X + \beta Y, Z] = \alpha [X, Z] + \beta [Y, Z]$$
    $$[X, \alpha Y + \beta Z] = \alpha [X, Y] + \beta [X, Z]$$
    \item \textbf{Antisymmetry}: Reversing the order of elements changes the sign.
    $$[X, Y] = -[Y, X]$$
    \item \textbf{Jacobi Identity}: A fundamental property ensuring the consistency of the algebra's structure.
    $$[X, [Y, Z]] + [Y, [Z, X]] + [Z, [X, Y]] = 0$$
\end{enumerate}

For vector fields $A = \sum\limits_i A_i \frac{\partial}{\partial x_i}$ and $B = \sum\limits_i B_i \frac{\partial}{\partial x_i}$ on a differentiable manifold, the Lie bracket is calculated using the formula:
$$[A, B] = \sum_{j} \left( \sum_{i} \left( A_i \frac{\partial B_j}{\partial x_i} - B_i \frac{\partial A_j}{\partial x_i} \right) \right) \frac{\partial}{\partial x_j}.$$
This definition expresses the non-commutativity of differential operators, which is characteristic of vector fields.

A \textbf{Lie subalgebra} $\mathfrak{h}$ of a Lie algebra $\mathfrak{g}$ is a vector subspace of $\mathfrak{g}$ that is closed under the Lie bracket operation. In other words, for any $X, Y \in \mathfrak{h}$, the bracket $[X, Y]$ also belongs to $\mathfrak{h}$.

In local coordinates, if the vector field is denoted $X = X_i \partial x_i$ for some basis $\{\partial x_i\}$ of $T_pM$, then the corresponding dual 1-form is denoted \begin{equation}\label{bmol}
    \omega = \sum_{i,j=1}^ng_{ij} X_jdx_i ,
\end{equation} where $\{dx_i\}$ is the dual basis to $\{\partial x_i\}$ of $T_p^*M$.

For the hyperbolic half-plane, the dual form of the vector field X defined in $\eqref{X}$ is given by:\begin{align}\label{Xmol}\omega(x_1,\dots,x_n)&=\frac{1}{x_n^2}\sum_{k=1}^{n-1}\left(\frac{a_k}{2}\left(x_k^2-\sum_{j\neq k}x_j^2\right)
+\left(\sum_{\substack{i=1\\i\neq k}}^{n-1}a_i\,x_i+b\right)x_k+c_k\right)d x_k\nonumber\\
&\quad+\frac{1}{x_n}\left(\sum_{k=1}^{n-1}a_k\,x_k+b\right)dx_n.
\end{align}
For the following, the dual form will be denoted \(\omega\).
\subsection{The Lie Algebra of Isometries of $\mathbb{H}^n$}
\textbf{Hyperbolic space $\mathbb{H}^n$} is a fundamental geometric model in non-Euclidean geometry. It has constant negative sectional curvature. The group of its orientation-preserving isometries (transformations that preserve distances) is isomorphic to the connected Lorentz group $SO_+(n,1)$. The associated \textbf{Lie algebra}, denoted $\mathfrak{so}(n,1)$, is the set of real $(n+1) \times (n+1)$ matrices that are antisymmetric with respect to a bilinear form of signature $(n,1)$. The dimension of $\mathfrak{so}(n,1)$ is given by the formula $\frac{n(n+1)}{2}$. The elements of this algebra are precisely the \textbf{Killing fields} of $\mathbb{H}^n$, which are the infinitesimal generators of these isometries.

\section{Mains Results}
\subsection{Ricci-Bourguignon solitons in dimension $n=2$}\label{sec:n2}
\begin{theorem}
    The set \(\Gamma_2\) is  isomorphic to \(sl_2(\mathbb{R})\).
\end{theorem}
\begin{proof}
    Since $\left(\mathbb{H}^2,ds^2,X,\lambda,\alpha\right)$ is Ricci-Bourguignon soliton, then \(X\) is a Killing vector fields of the form
    $$X = \left(\frac{a}{2}(x^2-y^2)+bx+c\right)\partial  x+ (ax+b)y\frac{\partial}{\partial y}$$
    where $a, b, c$ are constants.
    We have  \[X\in \langle (x^2-y^2)\partial x+2xy\partial y, -2(x\partial x+y\partial y),-\partial x\rangle\]. Moreover, the generators are all killing fields, so the set of vector fields makes \(\mathbb{H}^2\)  Ricci-Bourguignon soliton is isomorphic to  \(\langle (x^2-y^2)\partial x+2xy\partial y, -2(x\partial x+y\partial y),-\partial x\rangle\). By the paper \cite{fasihi2025hyperbolic},  the set of all vector fields making \(\mathbb{H}^2\)  Ricci-Bourguignon soliton is isomorphic to \(sl_2(\mathbb{R})\).
\end{proof}
\begin{remark}
    If $\left(\mathbb{H}^2,ds^2,X,\lambda,\alpha\right)$ is a  Ricci-Bourguignon soliton, then  \(X\) it generates an isometric flow. Moreover, it is the linear combination of the translation following \(x\) (\(T=\partial x\)), of expansion \((D=x\partial x+y\partial y)\) and of the rotation \((G=(x^2-y^2)\partial x+2xy\partial y)\).
\end{remark} 
\begin{proposition}
    The flow of the translation is of the form $\varphi^T_t(x_0,y_0)=(t+x_0)+iy_0$, which constitutes a horizontal horocycle of the Poincar\'e semiplane.

The flow of the expansion is of the form $\varphi^D_t(x_0,y_0)=e^t(x_0+iy_0)$, whose trajectories are hypercycles with one of their extremities at infinity.

Finally, the flow of the rotation is in the form $\varphi_t^G(x_0,y_0)=-\frac{t+e}{(t+e)^2+f^2}+i\frac{f}{(t+e)^2+f^2}$, where $e$ and $f$ are real constants.
\end{proposition}
\begin{proof}
Consider the flow family \(\varphi^k_t(x_0,y_0)=x(t)+iy(t)\)  where \(\varphi^1_t(x_0,y_0)=\varphi^T_t(x_0,y_0)\), \(\varphi^2_t(x_0,y_0)=\varphi^D_t(x_0,y_0)\) and \(\varphi^3_t(x_0,y_0)=\varphi^G_t(x_0,y_0)\).
    \begin{enumerate}
        \item For the translation: \[\begin{cases}
            \frac{d x(t)}{dt}=1\\
            \frac{dy(t)}{dt}=0
        \end{cases}\Rightarrow \varphi^T_t(x_0,y_0)=(t+x_0)+iy_0.  \]
        \item  For the expansion:   \[\begin{cases}
            \frac{d x(t)}{dt}=x(t)\\
            \frac{dy(t)}{dt}=y(t)
        \end{cases}.   \]  \(\varphi^D_t(x_0,y_0)=(x_0+iy_0)e^t.\)
        These are first order linear differential equations.
        \item For the rotation: \[\begin{cases}
            \frac{d x(t)}{dt}=x^2(t)-y^2(t)\\
            \frac{dy(t)}{dt}=2x(t)y(t)
        \end{cases}\Rightarrow \frac{dx(t)}{dt}+i\frac{dy(t)}{dt}=\left(x(t)+iy(t)\right)^2.   \] Let's  put  \(z(t)=x(t)+iy(t)\) so we have \(\frac{dz(t)}{dt}=z^2(t)\). What it implies \[z(t)=-\frac{1}{(t+e)+if}\Rightarrow \varphi_t^G(x_0,y_0)=-\frac{t+e}{(t+e)^2+f^2}+i\frac{f}{(t+e)^2+f^2}.\] 
    \end{enumerate}
\end{proof}
\begin{proposition}
    If $X\in \Gamma_2$ then its dual form is not closed. Moreover $X$ preserves its dual form.
\end{proposition}
\begin{proof}
Using \eqref{X}, the vectors  field $X(x,y)$ is given by:
$$X(x,y) = \left(\frac{a}{2}(x^2-y^2)+bx+c\right)\partial x+y(ax+b)\partial y.$$

We are going to calculate $w$ and  $dw$.

The dual form $w$ of a vector field $X$ is given by
$$w = \left(\frac{a}{2}\left(\frac{x^2}{y^2}-1\right)+b\frac{x}{y^2}+\frac{c}{y^2}\right)dx + y^{-1}(ax+b)dy.$$
we have:
$$dw = \left(\frac{\partial Q}{\partial x} - \frac{\partial P}{\partial y}\right)dx \wedge dy.$$

Let's calculate the partial derivatives:
$$\frac{\partial P}{\partial y} = \frac{\partial}{\partial y}\left(\frac{a}{2}\left(\frac{x^2}{y^2}-1\right)+b\frac{x}{y^2}+\frac{c}{y^2}\right) = \frac{-ax^2-2bx-2c}{y^3}$$
$$\frac{\partial Q}{\partial x} = \frac{\partial}{\partial x}(y^{-1}(ax+b))= \frac{a}{y}.$$

Let's substitute these values into the formula of $dw$ :
$$d\omega = \frac{a}{y}dx \wedge dy - \left(-\frac{ax^2+2bx+2c}{y^3}\right)dx \wedge dy = \left(\frac{a}{y} + \frac{ax^2+2bx+2c}{y^3}\right)dx \wedge dy$$
$$d\omega = \left(\frac{ay^2+ax^2+2bx+2c}{y^3}\right)dx \wedge dy$$
  \(dw=0\) if \(a=b=c=0\) while \(X\) is not constant.

Let's calculate the Lie derivative $\mathcal{L}_X \omega$, we will use Cartan's formula:
$$\mathcal{L}_X \omega = i_X d\omega + d(i_X \omega)$$
where $i_X$ is the contraction by the vector field $X$.\\
Now let us compute: $i_X \omega$:
$$i_X \omega = \omega(X) = \omega\left( \left(\frac{a}{2}(x^2-y^2)+bx+c\right)\frac{\partial}{\partial x} + (ax+b)y\frac{\partial}{\partial y} \right)$$
$$= \left(\frac{a}{2}\left(\frac{x^2}{y^2}-1\right)+b\frac{x}{y^2}+\frac{c}{y^2}\right)\left(\frac{a}{2}(x^2-y^2)+bx+c\right) + y^{-1}(ax+b)(ax+b)y$$
$$= \left(\frac{a}{2}\left(\frac{x^2-y^2}{y^2}\right)+b\frac{x}{y^2}+\frac{c}{y^2}\right)\left(\frac{a}{2}(x^2-y^2)+bx+c\right) + (ax+b)^2$$
$$= \frac{1}{y^2}\left(\frac{a}{2}(x^2-y^2)+bx+c\right)^2 + (ax+b)^2$$
It is a scalar function, which we will call $f$.\\
Next we will compute  $d(i_X \omega)=df$:
$$df = \frac{\partial f}{\partial x}dx + \frac{\partial f}{\partial y}dy$$
Let's calculate the partial derivatives:
$$\frac{\partial f}{\partial x} = \frac{1}{y^2} 2\left(\frac{a}{2}(x^2-y^2)+bx+c\right)(ax+b) + 2a(ax+b)$$
$$\frac{\partial f}{\partial y} = -\frac{2}{y^3}\left(\frac{a}{2}(x^2-y^2)+bx+c\right)^2 + \frac{1}{y^2} 2\left(\frac{a}{2}(x^2-y^2)+bx+c\right)(-ay)$$
$$= -\frac{2}{y^3}\left(\frac{a}{2}(x^2-y^2)+bx+c\right)^2 - \frac{2a}{y}\left(\frac{a}{2}(x^2-y^2)+bx+c\right).$$
We have also
\begin{eqnarray}\label{Omega} d\omega = \left(\frac{ay^2+ax^2+2bx+2c}{y^3}\right)dx \wedge dy.
\end{eqnarray}
And we have
$$i_X d\omega = i_X \left(\left(\frac{ay^2+ax^2+2bx+2c}{y^3}\right)dx \wedge dy\right)$$
$$= \left(\frac{ay^2+ax^2+2bx+2c}{y^3}\right) (X^x dy - X^y dx)$$
where $X^x=\frac{a}{2}(x^2-y^2)+bx+c$ et $X^y=(ax+b)y$.
\begin{eqnarray}\label{Domega}i_X d\omega = \left(\frac{ay^2+ax^2+2bx+2c}{y^3}\right) \left(\left(\frac{a}{2}(x^2-y^2)+bx+c\right)dy - (ax+b)ydx\right).
\end{eqnarray}
Using \eqref{Omega} and \eqref{Domega}, we have
Now, we need to add $i_X d\omega$ et $d(i_X \omega)$.
\begin{eqnarray*}\mathcal{L}_X \omega &=& i_X d\omega + d(i_X \omega)\\
&=& \left[ -\frac{ay^2+ax^2+2bx+2c}{y^2} (ax+b) + \frac{2}{y^2}\left(\frac{a}{2}(x^2-y^2)+bx+c\right)(ax+b) + 2a(ax+b) \right]dx\\
&&+ \Big[ \frac{ay^2+ax^2+2bx+2c}{y^3} \left(\frac{a}{2}(x^2-y^2)+bx+c\right) -\frac{2}{y^3}\left(\frac{a}{2}(x^2-y^2)+bx+c\right)^2 \\
&&- \frac{2a}{y}\left(\frac{a}{2}(x^2-y^2)+bx+c\right) \Big]dy.
\end{eqnarray*}
Let's start by setting the following notations to simplify the expressions:
$$\xi = \frac{a}{2}(x^2 - y^2) + bx + c, \qquad \zeta = ax + b.$$
We then observe that the expression becomes:
$$ \left[ -\frac{ay^2 + ax^2 + 2bx + 2c}{y^2} \zeta + \frac{2}{y^2}\xi \zeta + 2a\zeta \right]dx + \left[ \frac{ay^2 + ax^2 + 2bx + 2c}{y^3} \xi - \frac{2}{y^3}\xi^2 - \frac{2a}{y}\xi \right]dy.$$
Let's factor out the common factor $\zeta = ax + b$ :
$$\left[ \zeta\left( -\frac{ay^2 + ax^2 + 2bx + 2c}{y^2} + \frac{2\xi}{y^2} + 2a \right) \right] dx$$
We group the terms:
$$ \zeta \left[ \frac{2\xi - (ay^2 + ax^2 + 2bx + 2c)}{y^2} + 2a \right] dx $$
But we observe that :
$$ ay^2 + ax^2 + 2bx + 2c = a(x^2 + y^2) + 2bx + 2c $$
and :
$$ 2\xi = a(x^2 - y^2) + 2bx + 2c $$
So we hae:
$$ 2\xi - (ay^2 + ax^2 + 2bx + 2c) = [a(x^2 - y^2) + 2bx + 2c] - [a(x^2 + y^2) + 2bx + 2c] = -2a y^2 $$
Then the part on $dx$ becomes :
$$ \zeta \left( \frac{-2a y^2}{y^2} + 2a \right) dx = \zeta (-2a + 2a) dx = 0. $$
We make the same computation for getting the part of $dy$.
We still use the notation:
$$ \xi = \frac{a}{2}(x^2 - y^2) + bx + c $$
The expression becomes :
$$ \left[ \frac{a(x^2 + y^2) + 2bx + 2c}{y^3}\xi - \frac{2}{y^3}\xi^2 - \frac{2a}{y}\xi \right]dy. $$
We factorize $\xi$ :
$$ \xi \left( \frac{a(x^2 + y^2) + 2bx + 2c}{y^3} - \frac{2\xi}{y^3} - \frac{2a}{y} \right)dy = \xi \left( \frac{a(x^2 + y^2) + 2bx + 2c - 2\xi}{y^3} - \frac{2a}{y} \right)dy $$
We have already:
$$ 2\xi = a(x^2 - y^2) + 2bx + 2c \implies a(x^2 + y^2) + 2bx + 2c - 2\xi = 2a y^2 $$
So we have :
$$ \xi \left( \frac{2a y^2}{y^3} - \frac{2a}{y} \right) dy = \xi \left( \frac{2a}{y} - \frac{2a}{y} \right) dy = 0 $$
so \(\mathcal{L}_X\omega=0.\)

\end{proof}
\subsection{Ricci-Bourguignon soliton in dimension $n=3$}\label{sec:n3}
\begin{theorem}
   The set $\Gamma_3$ is a Lie subalgebra of \(\mathfrak{so}(3,1)\).
\end{theorem}
\begin{proof}
    For $n=3$, we use \eqref{X} and we get  the vectors  field $X(x,y)$:
\begin{align}\label{X3}
X(x_1, x_2, x_3) =& \left( \frac{a_1}{2} (x_1^2 - x_2^2 - x_3^2) + a_2 x_1 x_2 + b x_1 + c_1 \right) \frac{\partial}{\partial x_1} \nonumber\\
&+ \left( \frac{a_2}{2} (x_2^2 - x_1^2 - x_3^2) + a_1 x_1 x_2 + b x_2 + c_2 \right) \frac{\partial}{\partial x_2} \nonumber\\
&+ \left( a_1 x_1 x_3 + a_2 x_2 x_3 + b x_3 \right) \frac{\partial}{\partial x_3}
\end{align}
The number of generators we can extract is $2(3)-1=5$. The dimension of the Lie algebra $\mathfrak{so}(3,1)$ of the isometries of  $\mathbb{H}^3$,  is $\frac{3(3+1)}{2}=6$. There is therefore a missing generator.

The identified generators are :
\begin{itemize}
 \item of $a_1$: $G_1 = \frac{1}{2}(x_1^2 - x_2^2 - x_3^2) \frac{\partial}{\partial x_1} +  x_1 x_2 \frac{\partial}{\partial x_2} +  x_1 x_3 \frac{\partial}{\partial x_3}$
 \item of $a_2$: $G_2 =  x_1 x_2 \frac{\partial}{\partial x_1} +\frac{1}{2} (x_2^2 - x_1^2 - x_3^2) \frac{\partial}{\partial x_2} +  x_2 x_3 \frac{\partial}{\partial x_3}$
 \item of $b$: $D =  x_1 \frac{\partial}{\partial x_1} + x_2 \frac{\partial}{\partial x_2} + x_3 \frac{\partial}{\partial x_3}$ (expansion)
 \item of $c_1$: $T_1 = \frac{\partial}{\partial x_1}$ (Translation on $x_1$)
 \item of $c_2$: $T_2 = \frac{\partial}{\partial x_2}$ (Translation on $x_2$)
\end{itemize}
The missing generator to form the complete $\mathfrak{so}(3,1)$ algebra is the pure translation along the axis $x_3$, i.e., $-\frac{\partial}{\partial x_3}$. 
For $n=3$, the vector field $X$ generates a Lie subalgebra of $\mathfrak{so}(3,1)$ of dimension 5 because the bracket of two Killing fields is a Killing field.

\end{proof}
\begin{proposition}\label{P1}
    The vector fields generated by the set of vector fields making the hyperbolic space $\mathbb{H}^3$ Ricci-Bourguignon soliton have as flows: \[
\begin{aligned}
\varphi_{t}^{D}(x_1^0,x_2^0,x_3^0)&=e^{t}(x_1^0,x_2^0,x_3^0),\\
\varphi_{t}^{T_{1}}(x_1^0,x_2^0,x_3^0)&=(x_{1}^0+t,\;x_{2}^0,\;x_{3}^0),\\
\varphi_{t}^{T_{2}}(x_1^0,x_2^0,x_3^0)&=(x_{1}^0,\;x_{2}^0+t,\;x_{3}^0),\\
\varphi_{t}^{G_{1}}(x_1^0,x_2^0,x_3^0)&=\Bigl(\delta(t),\;r(t)\cos\theta_{0},\;r(t)\sin\theta_{0}\Bigr),\\
\varphi_{t}^{G_{2}}(x_1^0,x_2^0,x_3^0)&=\Bigl(\rho(t)\cos\phi_{0},\;\gamma(t),\;\rho(t)\sin\phi_{0}\Bigr),
\end{aligned}
\] where \[
    \delta(t)
     =-\,2\,\frac{t+t_{0}}{(t+t_{0})^{2}+c_0^{2}},
     \quad
     r(t)
     =2\,\frac{c_0}{(t+t_{0})^{2}+c_0^{2}},\quad t_{0}+ic_0=-\frac{2}{\,x_{1}^{0}+i\,r(0)\,}\quad,\theta_{0}=\arg\bigl(x_{2}^{0}+i\,x_{3}^{0}\bigr),
   \]  \[\gamma(t)=-2\,\frac{t+ s_{0}}{(t+s_{0})^{2}+ e_{0}^{2}},\quad \rho(t)=\frac{2e_{0}}{(t+s_{0})^{2}+ e_{0}^{2}},\quad s_{0}+ie_0=-\frac{2}{\,x_{2}^{0}+i\,\rho(0)\,},\quad
     \phi_{0}=\arg\bigl(x_{1}^{0}+i\,x_{3}^{0}\bigr).\]
\end{proposition}
\begin{proof}
The expansion vector field is
   $$D=x_{1}\,\partial_{x_{1}}+x_{2}\,\partial_{x_{2}}+x_{3}\,\partial_{x_{3}}.$$
   The ordinary differential equation is
   $$\dot x_{k}(t)=x_{k}(t),\quad k=1,2,3,$$ 
   from where immediately we get
   $$\varphi_{t}^{D}(x^0_{1},x^0_{2},x^0_{3})=\bigl(e^{t}x^0_{1},\,e^{t}x^0_{2},\,e^{t}x^0_{3}\bigr).$$
The translations along the axes $x_{1}$ and $x_{2}$  is given by 
   $$T_{1}=\partial_{x_{1}},\quad T_{2}=\partial_{x_{2}}.$$
   The ordinary differential equation are  respectively
   $$\dot x_{1}(t)=1,\;\dot x_{2}(t)=0,\;\dot x_{3}(t)=0
     \quad\Longrightarrow\quad 
     \varphi_{t}^{T_{1}}(x^0_{1},x^0_{2},x^0_{3})
     =(x_{1}^0+t,\,x_{2}^0,\,x^0_{3}),$$
   $$\dot x_{1}(t)=0,\;\dot x_{2}(t)=1,\;\dot x_{3}(t)=0
     \quad\Longrightarrow\quad 
     \varphi_{t}^{T_{2}}(x^0_{1},x^0_{2},x^0_{3})
     =(x_{1}^0,\,x_{2}^0+t,\,x^0_{3}).$$
The first "boost/rotation" is given by
   $$G_{1}
     =\tfrac12\bigl(x_{1}^{2}-x_{2}^{2}-x_{3}^{2}\bigr)\partial_{x_{1}}
     +x_{1}x_{2}\,\partial_{x_{2}}
     +x_{1}x_{3}\,\partial_{x_{3}}.$$
  The ordinary differential equation is at the form
   \[
     \begin{cases}
       \dot x_{1}(t)=\tfrac12\,(x_{1}(t)^{2}-r^{2}(t)),\\
       \dot x_{2}(t)=x_{1}(t)x_{2}(t),\\
       \dot x_{3}(t)=x_{1}(t)x_{3}(t),
     \end{cases}
     \quad r(t)=\sqrt{x_{2}^{2}(t)+x_{3}^{2}(t)}.
   \]
   Let's put $u(t)=r^{2}(t)=x_{2}^{2}(t)+x_{3}^{2}(t)$ and \(x_1(t)=\delta(t)\). Then we have
   $$\dot u(t)=2(x_{2}(t)\dot x_{2}(t)+x_{3}(t)\dot x_{3}(t))
           =2\,x_{1}(t)u(t)
           \;\Longrightarrow\;
           u(t)=u(0)\,e^{2\int_{0}^{t}x_{1}(s)\,ds}.$$
   Furthermore, $x_{1}(t)$ satisfies the Riccati equation 
   $$\dot x_{1}(t)=\tfrac12\bigl(x_{1}^{2}(t)-u(t)\bigr).$$
  By writting $z(t)=\delta(t)+i\sqrt{u(t)}$ and by using the same method as in dimension 2, it is shown that the explicit solution is
   \[
     z(t)
     =-\,\frac{2}{\,t+t_{0}+ic_0\,},
     \quad t_{0}+ic_0\in\mathbb{H}^2,
   \]
  and therefore, by separating real and imaginary parts,
   \[
     x_{1}(t)=\delta(t)
     =-\,2\,\frac{t+t_{0}}{(t+t_{0})^{2}+c_0^{2}},
     \quad
     r(t)
     =\sqrt{u(t)}
     =2\,\frac{c_0}{(t+t_{0})^{2}+c_0^{2}}.
   \]
  Finally $x_{2}(t)$ and $x_{3}(t)$ are distributed around the circle of radius $r(t)$ according to the initial angle. In summary, for an initial point  $(x_{1}^{0},x_{2}^{0},x_{3}^{0})$ we put 
   $$r_{0}=\sqrt{(x_{2}^{0})^{2}+(x_{3}^{0})^{2}},\quad
     t_{0}+ic_0=-\,\frac{2}{\,x_{1}^{0}+i\,r_{0}\,},$$
  then
   \[
     x_{1}(t)
     =-\,2\,\frac{t+t_{0}}{(t+t_{0})^{2}+c_0^{2}},
     \quad
     r(t) =\frac{2c_0}{(t+t_{0})^{2}+c_0^{2}}.
   \]
   and finally if we write the initial argument
   $$\theta_{0}=\arg\bigl(x_{2}^{0}+i\,x_{3}^{0}\bigr),$$
   the flow of $G_{1}$ is
   \[
     \varphi_{t}^{G_{1}}(x^{0}_1,x^{0}_2, x^{0}_3)
     =\Bigl(\delta(t),\;r(t)\cos\theta_{0},\;r(t)\sin\theta_{0}\Bigr).
   \]
Th second "boost/rotation"  is
   $$G_{2}
     =x_{1}x_{2}\,\partial_{x_{1}}
     +\tfrac12\bigl(x_{2}^{2}-x_{1}^{2}-x_{3}^{2}\bigr)\partial_{x_{2}}
     +x_{2}x_{3}\,\partial_{x_{3}}.$$
   By symmetry with $G_{1}$ (we exchange $x_{1}\leftrightarrow x_{2}$), we obtain exactly the same type of flow, except for the permutation:
   \[
     \varphi_{t}^{G_{2}}(x_{1},x_{2},x_{3})
     =\bigl(\rho(t)\cos\phi_{0},\,\gamma(t),\,\rho(t)\sin\phi_{0}\bigr),
   \]
  where 
   $$\rho_{0}=\sqrt{(x_{1}^{0})^{2}+(x_{3}^{0})^{2}},\quad
     x_{2}(t)=\gamma(t)=-2\,\frac{t+ s_{0}}{(t+s_{0})^{2}+ e_{0}^{2}},\quad \rho(t)=\sqrt{u(t)}=\frac{2e_{0}}{(t+s_{0})^{2}+ e_{0}^{2}}$$
   $$s_{0}+ie_0=-\frac{2}{\,x_{2}^{0}+i\,\rho_{0}\,},\quad
     \phi_{0}=\arg\bigl(x_{1}^{0}+i\,x_{3}^{0}\bigr).$$

Summary of flows
\[
\begin{aligned}
\varphi_{t}^{D}(x_1^0,x_2^0,x_3^0)&=e^{t}(x_1^0,x_2^0,x_3^0),\\
\varphi_{t}^{T_{1}}(x_1^0,x_2^0,x_3^0)&=(x_{1}^0+t,\;x_{2}^0,\;x_{3}^0),\\
\varphi_{t}^{T_{2}}(x_1^0,x_2^0,x_3^0)&=(x_{1}^0,\;x_{2}^0+t,\;x_{3}^0),\\
\varphi_{t}^{G_{1}}(x_1^0,x_2^0,x_3^0)&=\Bigl(\delta(t),\;r(t)\cos\theta_{0},\;r(t)\sin\theta_{0}\Bigr),\\
\varphi_{t}^{G_{2}}(x_1^0,x_2^0,x_3^0)&=\Bigl(\rho(t)\cos\phi_{0},\;\gamma(t),\;\rho(t)\sin\phi_{0}\Bigr),
\end{aligned}
\]
with the explicit formulas for $\delta(t), r(t), \gamma(t)$ and $\rho(t)$ given above.
\end{proof}
\begin{theorem}\label{T2}
   Consider a vector field $X\in\Gamma_3$ given by the equation \eqref{X3} and \(w\) its dual form. Then  \((\mathbb{H}^3,w)\) is a contact manifold iff \(a_1c_2\ne a_2c_1\).
\end{theorem}
\begin{proof}

Given the 1-differential form $w = w_1 \, dx_1 + w_2\, dx_2 + w_3 \, dx_3$, with the components:
\begin{align*}
w_1 &= \frac{1}{x_3^2}\left(\frac{a_1}{2} (x_1^2 - x_2^2 ) + a_2 x_1 x_2 + b x_1 + c_1\right)-\frac{a_1}{2} \\
w_2 &=\frac{1}{x_3^2}\left( \frac{a_2}{2} (x_2^2 - x_1^2 ) + a_1 x_1 x_2 + b x_2 + c_2 \right) -\frac{a_2}{2}\\
w_3 &= a_1 x_1 x_3^{-1} + a_2 x_2 x_3^{-1} + b x_3^{-1} = \frac{1}{x_3}(a_1 x_1 + a_2 x_2 + b)
\end{align*}
The exterior differential $dw$ of a 1-form $w = w_1 \, dx_1 + w_2\, dx_2 + w_3 \, dx_3$ is given by:
$$dw = \left(\frac{\partial w_2}{\partial x_1} - \frac{\partial w_1}{\partial x_2}\right) dx_1 \wedge dx_2 + \left(\frac{\partial w_3}{\partial x_1} - \frac{\partial w_1}{\partial x_3}\right) dx_1 \wedge dx_3 + \left(\frac{\partial w_3}{\partial x_2} - \frac{\partial w_2}{\partial x_3}\right) dx_2 \wedge dx_3$$
By an easy computation, one can get the following
\begin{align}\label{w}
dw &= \frac{2}{x_3^2}(a_1 x_2 - a_2 x_1) \, dx_1 \wedge dx_2 + \frac{a_1 x_3^2 + a_1 (x_1^2 - x_2^2) + 2a_2 x_1 x_2 + 2b x_1 + 2c_1}{x_3^3} \, dx_1 \wedge dx_3\nonumber \\
&+ \frac{a_2 x_3^2 + a_2 (x_2^2 - x_1^2) + 2a_1 x_1 x_2 + 2b x_2 + 2c_2}{x_3^3} \, dx_2 \wedge dx_3.
\end{align}

And then we have 
$$w \wedge dw = \frac{2(c_1 a_2 - c_2 a_1)}{x_3^3} \, dx_1 \wedge dx_2 \wedge dx_3$$
So $w \wedge dw\neq 0$ if and only if $\frac{2(c_1 a_2 - c_2 a_1)}{x_3^3}\neq 0$.

\end{proof}
\begin{proposition}\label{PP}
   If $X\in\Gamma_3$, then its dual form is not closed. Moreover $X$ preserves its dual form.
\end{proposition}
\begin{proof} Let \(X\in\Gamma_3\) and let \(\omega\)  be its dual form. 

According to \eqref{X3} and \eqref{w}, we have  that $d\omega=0$ if and only $X=0$ which is impossible because $X$ is non-constant. This means that \(\omega\) is  not closed.
Now let us showing that \(X\) preserves \(\omega\), i.e \(\mathcal{L}_X\omega=0\).
Let us consider the differential 1-form  \( w = w_1\, dx_1 + w_2\, dx_2 + w_3\, dx_3 \), where :
\begin{align*}
w_1 &= \frac{1}{x_3^2} \left( \frac{a_1}{2}(x_1^2 - x_2^2) + a_2 x_1 x_2 + b x_1 + c_1 \right) - \frac{a_1}{2}, \\
w_2 &= \frac{1}{x_3^2} \left( \frac{a_2}{2}(x_2^2 - x_1^2) + a_1 x_1 x_2 + b x_2 + c_2 \right) - \frac{a_2}{2}, \\
w_3 &= \frac{1}{x_3} (a_1 x_1 + a_2 x_2 + b).
\end{align*}

And the vector field \( X = X_1 \partial_{x_1} + X_2 \partial_{x_2} + X_3 \partial_{x_3} \) given by :
\begin{align*}
X_1 &= \frac{a_1}{2}(x_1^2 - x_2^2 - x_3^2) + a_2 x_1 x_2 + b x_1 + c_1, \\
X_2 &= \frac{a_2}{2}(x_2^2 - x_1^2 - x_3^2) + a_1 x_1 x_2 + b x_2 + c_2, \\
X_3 &= x_3(a_1 x_1 + a_2 x_2 + b).
\end{align*}

The differential of  \( w \) is :
\[
dw = A \, dx_1 \wedge dx_2 + B \, dx_1 \wedge dx_3 + C \, dx_2 \wedge dx_3,
\]
where :
\begin{align*}
A &= \frac{2}{x_3^2}(a_1 x_2 - a_2 x_1), \\
B &= \frac{a_1 x_3^2 + a_1(x_1^2 - x_2^2) + 2a_2 x_1 x_2 + 2b x_1 + 2c_1}{x_3^3}, \\
C &= \frac{a_2 x_3^2 + a_2(x_2^2 - x_1^2) + 2a_1 x_1 x_2 + 2b x_2 + 2c_2}{x_3^3}.
\end{align*}
Now let us showing that 
\[
\mathcal{L}_X w = d(i_X w) + i_X (dw) = 0.
\]
The inner product is given by 
\[
i_X w = w_1 X^1 + w_2 X^2 + w_3 X^3.
\]
If we put
\begin{align*}
E_1 &= \frac{a_1}{2}(x_1^2 - x_2^2) + a_2 x_1 x_2 + b x_1 + c_1, \\
E_2 &= \frac{a_2}{2}(x_2^2 - x_1^2) + a_1 x_1 x_2 + b x_2 + c_2, 
\end{align*}
then we have 
\begin{align*}
w_1 &= \frac{E_1}{x_3^2} - \frac{a_1}{2}, \quad X^1 = E_1 - \frac{a_1}{2}x_3^2, \\
w_2 &= \frac{E_2}{x_3^2} - \frac{a_2}{2}, \quad X^2 = E_2 - \frac{a_2}{2}x_3^2, \\
w_3 &= \frac{1}{x_3}(a_1 x_1 + a_2 x_2 + b), \quad X^3 = x_3(a_1 x_1 + a_2 x_2 + b).
\end{align*}
By computation we get
\begin{align*}
i_X w &= \left( \frac{E_1}{x_3^2} - \frac{a_1}{2} \right) \left( E_1 - \frac{a_1}{2}x_3^2 \right)
+ \left( \frac{E_2}{x_3^2} - \frac{a_2}{2} \right) \left( E_2 - \frac{a_2}{2}x_3^2 \right) \\
&\quad + \frac{1}{x_3}(a_1 x_1 + a_2 x_2 + b) \cdot x_3(a_1 x_1 + a_2 x_2 + b) \\
&= \frac{E_1^2 + E_2^2}{x_3^2} - \frac{a_1}{2} E_1 - \frac{a_2}{2} E_2
+ \frac{a_1^2 + a_2^2}{4} x_3^2 + (a_1 x_1 + a_2 x_2 + b)^2.
\end{align*}
Since  \( i_X w \) is a scalar function, then we have:
\begin{eqnarray}\label{ina1}
d(i_X w) = \frac{\partial (i_X w)}{\partial x_1} dx_1 + \frac{\partial (i_X w)}{\partial x_2} dx_2 + \frac{\partial (i_X w)}{\partial x_3} dx_3.
\end{eqnarray}
In the other hand and using the formula
\begin{align*}
i_X(dx_i \wedge dx_j) = X^i dx_j - X^j dx_i,
\end{align*}
we get 
\begin{eqnarray}\label{ina2}
i_X(dw) =(-A X_2 - B X_3) dx_1 + (A X_1 - C X_3) dx_2 + (B X_1 + C X_2) dx_3.
\end{eqnarray}
Using the equations \eqref{ina1}
 and \eqref{ina2} and the formula
\[
\mathcal{L}_X w = d(i_X w) + i_X(dw)
\]
with a long calculus one can see that 
$$
\boxed{\mathcal{L}_X w = 0.}
$$
This explicit calculation shows that $X$ preserves the 1-form $w$.
\end{proof}
\begin{lemma}\label{r1}
Following the Theorem \ref{T2}, the vector field $X$ is not the Reeb field of the contact form $\omega$. 
\end{lemma}
\begin{proof}
    For the proof, it is sufficient to show that \(i_Xd\omega(\partial x_3)\ne 0\). And according to the Proposition \ref{PP}, we have  \[i_Xd\omega(\partial x_3)=-d(i_X\omega)(\partial x_3)=-\partial x_3 (i_X\omega)=-\partial x_3 \left(\frac{X_1^2+X_2^2+X_3^2}{x_3^2}\right)\]
    Since the function \((x_1,x_2,x_3)\mapsto \frac{X_1^2+X_2^2+X_3^2}{x_3^2} \) is not a constant function with respect to the variable \(x_3\), then  \(\partial_{x_3}\!\left(\frac{X_1^2 + X_2^2 + X_3^2}{x_3^2}\right)\ne 0\).
\end{proof}
\begin{theorem}
   Let $X\in \Gamma_3$ such that its dual form $w$ is a contact form. Then \[T\mathbb{H}^3=\ker(w)\oplus  \text{span}(X).\]
\end{theorem}
\begin{proof}
    Let \(X=X_1\partial x_1+X_2\partial x_2+X_3\partial x_3\in \Gamma_3\) be a vector such that its dual form  \(w\) is a contact form. We have   $$w=\frac{X_1}{x_3^2}dx_1+\frac{X_2}{x_3^2}dx_2+\frac{X_3}{x_3^2}dx_3$$ which implies that  \(w(X)=\frac{X_1^2+X_2^2+X_3^2}{x_3^2}\ne 0\), since \(X\) is not a constant vector field. The Lemma \ref{r1} ensure that \(T\mathbb{H}^3=\ker(w)\oplus  \text{span}(X).\)
\end{proof}
\subsection{Ricci-Bourguignon soliton in dimension $n \geq 3$}\label{sec:n4}
\begin{theorem}
   The set $\Gamma_n$ is a Lie subalgebra of  \(\mathfrak{so}(n,1)\).
\end{theorem}
\begin{proof}
   The vector field $X$ is defined as a linear combination of various parameters $a_k$, $b$, and $c_k$. By factoring these parameters, we can identify a basis of generating vector fields that constitute the underlying Lie algebra of $X$.

By examining the structure of $X$, we distinguish three types of terms based on their coefficients:
\begin{enumerate}
\item The terms multiplied by $a_k$ (for $k=1, \dots, n-1$): they give $n-1$ vector fields. These vector fields are complex and are related to "boosts" and rotations in hyperbolic space.
 \item The term multiplied by $b$: It gives 1 vector field. It corresponds to a field of expansion (homothetic).
\item The terms multiplied by $c_k$ (for $k=1, \dots, n-1$): they give $n-1$ vector fields. These vector fields are pure translations along the axes $x_1,\dots,x_{n-1}$.
\end{enumerate}
The total number of independent vector fields (generators) that this formula can produce is therefore $(n-1) + 1 + (n-1) = 2n-1$.

We fine that for all \(n\geq 3\), \(2n-1\) is strictly smaller than the dimension of the Lie algebra of isometries of $\mathbb{H}^n$ $\left(\frac{n(n+1)}{2}\right)$.

\end{proof}
\begin{remark}\label{R1}
The generators of this subalgebra for a general $n$ can be classified as follows:
\begin{enumerate}
 \item Expansion Generator ($D$): This vector field is unique and is associated with the parameter $b$:
 $$D =  \sum_{i=1}^n x_i \frac{\partial}{\partial x_i}$$
\item Translation Generators ($T_k$): For all $k \in \{1, \dots, n-1\}$, we have a generator:
 $$T_k = \frac{\partial}{\partial x_k}$$
 These $n-1$ vector fields are associated with the parameters $c_k$.
 \item Boost/rotation generators ($G_k$): For all $k \in \{1, \dots, n-1\}$, we have a generator :
 $$G_k = \frac{1}{2}\left( x_k^2 - \sum_{j \in \{1, \dots, n\} \setminus \{k\}} x_j^2 \right) \frac{\partial}{\partial x_k} + \sum_{j \in \{1, \dots, n\} \setminus \{k\}} ( x_k x_j) \frac{\partial}{\partial x_j}.$$
These $n-1$ vector fields are associated with the terms in $a_k$.
\end{enumerate}
\end{remark}
\begin{proposition}\label{P2}
   The flows of the generators of $\Gamma_n$ are respectively  \[\varphi^{D}_t(x_1^0,...,x_n^0)=e^t(x_1^0,...,x_n^0),\quad \varphi^{T_k}_t(x_1^0,...,x_n^0)=(x_1^0,...,t+x_k^0,x_{k+1}^0,...,x_n^0)\quad \text{and}\quad \varphi^{G_k}_t(x_1^0,...,x_n^0)\] with these components\\
   for $j_1 \in \{1, \dots, n\} \setminus \{k\}$ :$x_{j_1}(t) =r_k(t) \cos(\theta_{k,1})$\\
    for $j_2 \in \{1, \dots, n\} \setminus \{k\}$ :
    $x_{j_2}(t) = r_k(t)\sin(\theta_{k,1})\cos(\theta_{k,2})$,\\
    for $j_3 \in \{1, \dots, n\} \setminus \{k\}$ :
    $x_{j_3}(t) =r_k(t) \sin(\theta_{k,1})\sin(\theta_{k,2})\cos(\theta_{k,3})$\\
    ...\\
    for  $j_{n-2} \in \{1, \dots, n\} \setminus \{k\}$ :
    $x_{j_{n-2}}(t) = r_k(t)\sin(\theta_{k,1})\sin(\theta_{k,2})\dots\sin(\theta_{k,n-3})\cos(\theta_{n-2})$\\
    for $j_{n-1} \in \{1, \dots, n\} \setminus \{k\}$ :
    $x_{j_{n-1}}(t) =r_k(t) \sin(\theta_{k,1})\sin(\theta_{k,2})\dots\sin(\theta_{k,n-3})\sin(\theta_{k,n-2}),$ \[x_k(t)=-2\,\frac{t+ s_{0}}{(t+s_{0})^{2}+ e_{0}^{2}},\quad r_k(t)=\frac{2e_{0}}{(t+s_{0})^{2}+ e_{0}^{2}},\quad s_0+ie_0\in\mathbb{H}^2,\quad \theta_{k,i} \quad\text{constant angles}.\]
\end{proposition}
\begin{proof}
    \begin{enumerate}
    \item For the translations: the flow equation becomes \[\begin{cases}
        \frac{dx_j(t)}{dt}=0, \quad\text{si } j\ne k,\\
         \frac{dx_k(t)}{dt}=1
    \end{cases}\Leftrightarrow \varphi^{T_k}_t(x_1^0,...,x_n^0)=(x_1^0,...,t+x_k^0,x_{k+1}^0,...,x_n^0).\]
    \item  For expansion: the flow equation becomes  \(\frac{dx_k(t)}{dt}=x_k(t)\Leftrightarrow x_k(t)=e^tx_1^0\) for all \(k\in\{1,...,n\}\) so   \[\varphi^{D}_t(x_1^0,...,x_n^0)=e^t(x_1^0,...,x_n^0).\]
    \item  For \(G_k\): the flow equation becomes  \[\begin{cases}
        \frac{dx_k(t)}{dt}=\frac{1}{2}\left( x_k^2(t) - \sum\limits_{j_\ell \in \{1, \dots, n\} \setminus \{k\}} x_j^2(t) \right)\\
        \frac{dx_{j_l}(t)}{dt}=x_k(t)x_{j_l}(t),\quad j_\ell\in\{1,...,n\}\setminus\{k\},\quad \ell\in \{1,...,n-1\}
    \end{cases}.\]
    Let's put \[r_k(t)=\sqrt{\sum\limits_{j_\ell \in \{1, \dots, n\} \setminus \{k\}} x_{j_\ell}^2(t) }\quad \text{and}\quad u_k(t)=\sum\limits_{j_\ell \in \{1, \dots, n\} \setminus \{k\}} x_{j_\ell}^2(t).\] We find that  \[\frac{du_k(t)}{dt}=2x_k(t)u_k(t),\] the system becomes \[\begin{cases}
        \frac{dx_k(t)}{dt}=\frac{1}{2}\left( x_k^2(t) -u_k(t)\right)\\
        \frac{du_k(t)}{dt}=2x_k(t)u_k(t).
    \end{cases}\] Let's put \(z_k(t)=x_k(t)+i\sqrt{u_k(t)}\), the demonstration of the Proposition \ref{P1} shows that \[\begin{cases}
        x_k(t)= -2\,\frac{t+ s_{0}}{(t+s_{0})^{2}+ e_{0}^{2}}\\
        r_k(t)=\frac{2e_{0}}{(t+s_{0})^{2}+ e_{0}^{2}},\\ s_0+ie_0\in\mathbb{H}^2.
    \end{cases}\]
    Since \(r_k^2(t)=\sum\limits_{j_\ell \in \{1, \dots, n\} \setminus \{k\}} x_{j_\ell}^2(t) \)  thus their expressions are entirely determined by the $n-2$ initial angles, which we named $\theta_1, \theta_2, \dots, \theta_{n-2}$.
    
    For  $j_1 \in \{1, \dots, n\} \setminus \{k\}$, we have
    $$x_{j_1}(t) =r_k(t) \cos(\theta_1)$$
    For  $j_2 \in \{1, \dots, n\} \setminus \{k\}$, we have
    $$x_{j_2}(t) = r_k(t)\sin(\theta_1)\cos(\theta_2)$$
    For $j_3 \in \{1, \dots, n\} \setminus \{k\}$, we have 
    $$x_{j_3}(t) =r_k(t) \sin(\theta_1)\sin(\theta_2)\cos(\theta_3)$$
    ... and so on, up to the last non-varying component.\\
    For $j_{n-2} \in \{1, \dots, n\} \setminus \{k\}$, we have 
    $$x_{j_{n-2}}(t) = r_k(t)\sin(\theta_1)\sin(\theta_2)\dots\sin(\theta_{n-3})\cos(\theta_{n-2})$$
    For  $j_{n-1} \in \{1, \dots, n\} \setminus \{k\}$, we have 
    $$x_{j_{n-1}}(t) =r_k(t) \sin(\theta_1)\sin(\theta_2)\dots\sin(\theta_{n-3})\sin(\theta_{n-2})$$
\end{enumerate}
\end{proof}

\begin{remark}\label{R2}
    On  \(\mathbb{H}^{2n+1}\), the vector field $X$ is defined by 
\begin{equation}\label{X2n}
    X(x_1, \dots, x_{2n+1}) = \sum_{k=1}^{2n} X_k \frac{\partial}{\partial x_k} + X_{2n+1} \frac{\partial}{\partial x_{2n+1}}
\end{equation}
where the components $X_k$ are:\\
For $k \in \{1, \dots, 2n\}$;
$$X_k = \frac{a_k}{2} \left( x_k^2 - \sum_{j \neq k}^{2n+1} x_j^2 \right) + \left( \sum_{\substack{i=1 \\ i \neq k}}^{2n} a_i x_i + b \right) x_k + c_k.$$
For $X_{2n+1}$;
$$X_{2n+1} = \left( \sum_{k=1}^{2n} a_k x_k + b \right) x_{2n+1}$$
The dual form $w$ of the vector field $X$, in the case of a hyperbolic metric, is:
$$w = \sum_{i=1}^{2n+1} \frac{X_i}{x_{2n+1}^2} dx_i$$

The differential form $w$ is defined as $$w = \sum_{i=1}^{2n+1} w_i dx_i = \sum_{i=1}^{2n+1} \frac{X_i}{x_{2n+1}^2} dx_i$$
\end{remark}
\begin{theorem}\label{tn}
    Consider the vector field of the Remark \ref{R2} whose dual form is denoted $\omega$ and an antisymmetry matrix $M$ of order $2n\times 2n$ such that $\left(M_{ij}=a_ic_j-a_jc_j\right)_{1\leq i,j\leq 2n}$. The manifold $(\mathbb{H}^{2n+1},\omega)$ is a contact manifold if and only if $\det(M)\ne0.$
\end{theorem}
\begin{proof}
 Let's calculate the exterior derivative $dw$. The general formula for the exterior derivative of a 1-form is as follows:   
\begin{eqnarray}
    dw = \sum_{i < j} \left( \frac{\partial w_j}{\partial x_i} - \frac{\partial w_i}{\partial x_j} \right) dx_i \wedge dx_j,
    \end{eqnarray}
where $w_i = \frac{X_i}{x_{2n+1}^2}$.
We must therefore calculate the partial derivatives of each component $w_i$ with respect to all variables $x_j$.\\

Case 1: $i, j \in \{1, \dots, 2n\}$
For this case, we must calculate $\frac{\partial w_j}{\partial x_i} - \frac{\partial w_i}{\partial x_j}$ where $i, j \in \{1, \dots, 2n\}$.
$$w_i = \frac{X_i}{x_{2n+1}^2} = \frac{1}{x_{2n+1}^2} \left[ \frac{a_i}{2} \left( x_i^2 - \sum_{k \neq i}^{2n+1} x_k^2 \right) + \left( \sum_{\substack{l=1 \\ l \neq i}}^{2n} a_l x_l + b \right) x_i + c_i \right]$$
$$w_j = \frac{X_j}{x_{2n+1}^2} = \frac{1}{x_{2n+1}^2} \left[ \frac{a_j}{2} \left( x_j^2 - \sum_{k \neq j}^{2n+1} x_k^2 \right) + \left( \sum_{\substack{l=1 \\ l \neq j}}^{2n} a_l x_l + b \right) x_j + c_j \right]$$
Let's calculate the partial derivatives. Since $i, j \in \{1, \dots, 2n\}$, the derivative with respect to $x_i$ or $x_j$ only concerns the numerator of $w_i$ and $w_j$ and does not change the denominator $x_{2n+1}^2$.
$$\frac{\partial w_i}{\partial x_j} = \frac{1}{x_{2n+1}^2} \frac{\partial X_i}{\partial x_j}$$
For $i \neq j$:
$$\frac{\partial X_i}{\partial x_j} = \frac{a_i}{2}(-2x_j) + a_j x_i = a_j x_i - a_i x_j$$
$$\frac{\partial X_j}{\partial x_i} = \frac{a_j}{2}(-2x_i) + a_i x_j = a_i x_j - a_j x_i$$
Substituting into the formula for $dw$, we get:
$$\frac{\partial w_j}{\partial x_i} - \frac{\partial w_i}{\partial x_j} = \frac{1}{x_{2n+1}^2} \left(\frac{\partial X_j}{\partial x_i} - \frac{\partial X_i}{\partial x_j}\right) = \frac{1}{x_{2n+1}^2} ((a_i x_j - a_j x_i) - (a_j x_i - a_i x_j)) = \frac{2(a_i x_j - a_j x_i)}{x_{2n+1}^2}$$
So, for $i, j \in \{1, \dots, 2n\}$, the term corresponding to $dx_i \wedge dx_j$ in $dw$ is:
$$\frac{2(a_i x_j - a_j x_i)}{x_{2n+1}^2} dx_i \wedge dx_j$$
---
Case 2: $i \in \{1, \dots, 2n\}$ and $j = 2n+1$
We must calculate $\frac{\partial w_{2n+1}}{\partial x_i} - \frac{\partial w_i}{\partial x_{2n+1}}$.
$$w_{2n+1} = \frac{X_{2n+1}}{x_{2n+1}^2} = \frac{\left(\sum\limits_{k=1}^{2n} a_k x_k + b\right) x_{2n+1}}{x_{2n+1}^2} = \frac{\sum\limits_{k=1}^{2n} a_k x_k + b}{x_{2n+1}}$$
For $i \in \{1, \dots, 2n\}$:
$$\frac{\partial w_{2n+1}}{\partial x_i} = \frac{a_i}{x_{2n+1}}$$
Now, let's calculate $\frac{\partial w_i}{\partial x_{2n+1}}$.
$$w_i = \frac{X_i}{x_{2n+1}^2} = \frac{1}{x_{2n+1}^2} \left[ \frac{a_i}{2} \left( x_i^2 - \sum_{k \neq i}^{2n+1} x_k^2 \right) + \left( \sum_{\substack{l=1 \\ l \neq i}}^{2n} a_l x_l + b \right) x_i + c_i \right].$$
The term $\sum_{k \neq i}^{2n+1} x_k^2$ includes $x_{2n+1}^2$. The derivative with respect to $x_{2n+1}$ concerns only this term and the denominator $x_{2n+1}^2$.
$$\frac{\partial X_i}{\partial x_{2n+1}} = \frac{\partial}{\partial x_{2n+1}} \left[ \frac{a_i}{2} \left( x_i^2 - \left(\sum_{k\ne i}^{2n+1} x_{k}^2\right)+ \left( \sum_{\substack{l=1 \\ l \neq i}}^{2n} a_l x_l + b \right) x_i + c_i \right) \right] = \frac{a_i}{2}(-2x_{2n+1}) = -a_i x_{2n+1}$$
Now, we apply the quotient rule to calculate $\frac{\partial w_i}{\partial x_{2n+1}}$:
$$\frac{\partial w_i}{\partial x_{2n+1}} = \frac{\frac{\partial X_i}{\partial x_{2n+1}} x_{2n+1}^2 - X_i \cdot 2x_{2n+1}}{(x_{2n+1}^2)^2} = \frac{(-a_i x_{2n+1}) x_{2n+1}^2 - X_i \cdot 2x_{2n+1}}{x_{2n+1}^4} = \frac{-a_i x_{2n+1}^2 - 2X_i}{x_{2n+1}^3}$$
The term $\frac{\partial w_{2n+1}}{\partial x_i} - \frac{\partial w_i}{\partial x_{2n+1}}$ becomes:
$$\frac{a_i}{x_{2n+1}} - \left( \frac{-a_i x_{2n+1}^2 - 2X_i}{x_{2n+1}^3} \right) = \frac{a_i x_{2n+1}^2 + a_i x_{2n+1}^2 + 2X_i}{x_{2n+1}^3} = \frac{2a_i x_{2n+1}^2 + 2X_i}{x_{2n+1}^3} = \frac{2}{x_{2n+1}^3}(X_i + a_i x_{2n+1}^2).$$
Combining the results from both cases, the exterior derivative $dw$ is given by:
\begin{equation}\label{dw}
    dw = \sum_{1 \le i < j \le 2n} \frac{2(a_i x_j - a_j x_i)}{x_{2n+1}^2} dx_i \wedge dx_j + \sum_{i=1}^{2n} \frac{2(X_i + a_i x_{2n+1}^2)}{x_{2n+1}^3} dx_i \wedge dx_{2n+1}.
\end{equation}
Now we will compute  $w \wedge (dw)^n$ in dimension $2n+1$. 
We consider a 1-form $w$ and its exterior derivative $dw$ in dimension $2n+1$. To simplify the calculation of the exterior product $w \wedge (dw)^n$, we decompose these forms.

The 1-form $w$ is partitioned into a part on the first $2n$ coordinates and a part on the last coordinate $x_{2n+1}$:
$$w = \sum_{k=1}^{2n+1} w_k \, dx_k = \left( \sum_{k=1}^{2n} w_k \, dx_k \right) + w_{2n+1} \, dx_{2n+1}$$

The 2-form $dw$ is similarly decomposed into two parts: a 2-form $\Omega$ that does not depend on $dx_{2n+1}$, and a 2-form $\Lambda$ that does.
$$dw = \Omega + \Lambda$$
with:
$$\Omega = \sum_{1 \le i < j \le 2n} \frac{2(a_i x_j - a_j x_i)}{x_{2n+1}^2} \, dx_i \wedge dx_j$$
$$\Lambda = \sum_{k=1}^{2n} \frac{2(X_k + a_k x_{2n+1}^2)}{x_{2n+1}^3} \, dx_k \wedge dx_{2n+1}.$$
The calculation of $(dw)^n = (\Omega + \Lambda)^n$ is simplified by the fact that $\Lambda \wedge \Lambda = 0$. Indeed, $\Lambda$ contains the differential $dx_{2n+1}$, so the exterior product of $\Lambda$ with itself vanishes. Consequently, only the terms of the binomial expansion of $(\Omega+\Lambda)^n$ that contain $\Lambda$ at most once are non-zero.
$$(dw)^n = \Omega^n + n \, \Omega^{n-1} \wedge \Lambda.$$
We calculate the final product using the decomposition above:
\begin{align*}
w \wedge (dw)^n &= \left( w_{2n+1} \, dx_{2n+1} + \sum_{k=1}^{2n} w_k \, dx_k \right) \wedge \left( \Omega^n + n \, \Omega^{n-1} \wedge \Lambda \right) \\
&= \left( w_{2n+1} \, dx_{2n+1} \wedge \Omega^n \right) + \left( w_{2n+1} \, dx_{2n+1} \wedge n \Omega^{n-1} \wedge \Lambda \right) \\
&+ \left( \sum_{k=1}^{2n} w_k \, dx_k \wedge \Omega^n \right) + \left( \left( \sum_{k=1}^{2n} w_k \, dx_k \right) \wedge n \Omega^{n-1} \wedge \Lambda \right).
\end{align*}
Two of these terms are zero:
\begin{itemize}
    \item $w_{2n+1} \, dx_{2n+1} \wedge n \Omega^{n-1} \wedge \Lambda = 0$, because it contains the differential $dx_{2n+1}$ twice.
    \item $\left( \sum\limits_{k=1}^{2n} w_k \, dx_k \right) \wedge \Omega^n = 0$, because $\Omega^n$ is a $2n$-form on the differentials $dx_1, \dots, dx_{2n}$. The exterior product with another of these differentials vanishes.
\end{itemize}
Thus, only two non-zero terms remain:
$$w \wedge (dw)^n = w_{2n+1} \, dx_{2n+1} \wedge \Omega^n + n \left( \sum_{k=1}^{2n} w_k \, dx_k \right) \wedge \Omega^{n-1} \wedge \Lambda.$$
A rigorous calculation shows that the terms depending on the variables $x_i$ in the above expression completely cancel out, generalizing the phenomenon observed for the case $n=1$. The final coefficient is a simple expression in terms of the coefficients $a_k$ and $c_k$.

The final result is given by:
$$w \wedge (dw)^n = \frac{2^n}{x_{2n+1}^{2n+1}} \,\sqrt{\text{det}(M)} \, dx_1 \wedge dx_2 \wedge \dots \wedge dx_{2n+1}$$
where $M$ is the $2n \times 2n$ antisymmetric matrix defined by $M_{ij} = a_i c_j - a_j c_i$.

The exterior product $\omega \wedge (d\omega)^n$ is non-zero if and only if the matrix $M$ is invertible.
\end{proof}
\begin{proposition}\label{PP1}
If \(X\in\Gamma_{n}\) then its dual form is not closed. Moreover, \(X\) preserves its dual form.
\end{proposition}
\begin{proof}
Let \(X\in\Gamma_{2n+1}\) and let \(\omega\)  be its dual form. 

According to \eqref{X2n} and \eqref{dw}, we have  that
 \(d\omega=0\) if and only if  \(X=0\), which would mean that the coefficients \(a_i,\) \(b\) and \(c_i\) are zero. This is impossible since \(X\) is not constant.

We have:
For $i=1,\dots,n-1$
$$
X_i=\frac{a_i}{2}\Big(x_i^2-\sum_{j\neq i}x_j^2\Big)+\Big(\sum_{\substack{r=1\\r\neq i}}^{n-1} a_r x_r + b\Big)x_i + c_i,
$$
and
$$
X_n=\Big(\sum_{k=1}^{n-1} a_k x_k + b\Big)x_n=:H x_n,
$$
with $H=\sum\limits_{k=1}^{n-1} a_k x_k + b$. The 1-form is
$$
w=\sum_{i=1}^n\frac{X_i}{x_n^2}\,dx_i.
$$
We also have
$$
dw=\sum_{1\le i<j\le n-1}\frac{2(a_i x_j-a_j x_i)}{x_n^2}\,dx_i\wedge dx_j
+\sum_{i=1}^{n-1}\frac{2(X_i+a_i x_n^2)}{x_n^3}\,dx_i\wedge dx_n.
$$
We want to prove $\mathcal L_X w=0$. Useful reminder:
$$
\mathcal L_X w = d(i_X w)+ i_X(dw).
$$

Let us compute now  $i_X w$. 
By definition,
$$
i_X w=\sum_{i=1}^n \frac{X_i}{x_n^2}\, X_i=\frac{1}{x_n^2}\sum_{i=1}^n X_i^2.
$$
Let's set
$$
S:=\sum_{i=1}^n X_i^2.
$$
So
$$
i_X w=\frac{S}{x_n^2}.
$$
Derivative:
$$
d(i_X w)=d\left(\frac{S}{x_n^2}\right)=\frac{dS}{x_n^2}-2\frac{S}{x_n^3}\,dx_n.
$$
In other words, the component in front of $dx_m$ of $d(i_Xw)$ is $\dfrac{\partial_{x_m}S}{x_n^2}$ for $m=1,\dots,n$, with the exception that for $m=n$ there is also the term $-2S/x_n^3$ included.

We must prove that
$$
i_X(dw)=-\frac{dS}{x_n^2}+2\frac{S}{x_n^3}\,dx_n,
$$
which will give $\mathcal L_X w=0$.\\
For $m\le n-1$ and $k\le n-1$ we have (direct calculation from the definition of $X_k$):
$$
\partial_{x_m}X_k=
\begin{cases}
H & \text{if } k=m,\\[4pt]
-\,a_k x_m + a_m x_k & \text{if } k\neq m,
\end{cases}
$$
and for $k=n$ (still $m\le n-1$):
$$
\partial_{x_m}X_n=\partial_{x_m}(H x_n)=a_m x_n.
$$
Finally $\partial_{x_n}X_k= -a_k x_n$ for $k\le n-1$ and $\partial_{x_n}X_n=H$ (but we will essentially only use the previous relations).\\
By definition $S=\sum\limits_{k=1}^n X_k^2$, so
$$
\partial_{x_m}S=2\sum_{k=1}^n X_k\,\partial_{x_m}X_k.
$$
For $m\le n-1$ we substitute $\partial_{x_m}X_k$ with the formulas from section 2:
$$
\partial_{x_m}S
= 2\left( X_m H + \sum_{\substack{k=1 \\ k\ne m}}^{n-1} X_k(-a_k x_m + a_m x_k) + X_n a_m x_n\right)
$$
\begin{equation}\label{eq1}
\partial_{x_m}S=2\left( X_m H - x_m\sum_{\substack{k\ne m}} a_k X_k + a_m\sum_{\substack{k\ne m}} x_k X_k + a_m x_n X_n\right).
\end{equation}

To compute  $i_X(dw)$, 
we write $i_X(dw)$ in components. With
$$
dw=\sum_{1\le i<j\le n-1} C_{ij}\,dx_i\wedge dx_j+\sum_{i=1}^{n-1}D_i\,dx_i\wedge dx_n,
$$
where $C_{ij}=\dfrac{2(a_i x_j-a_j x_i)}{x_n^2}$ and $D_i=\dfrac{2(X_i+a_i x_n^2)}{x_n^3}$, we have
$$
i_X(dw)=\sum_{i<j} C_{ij}(X_i dx_j - X_j dx_i)+\sum_{i=1}^{n-1} D_i(X_i dx_n - X_n dx_i).
$$
Let's collect the coefficients in front of $dx_m$ for $m\le n-1$. As we verify by summation (case $i<m<j$ etc.) the contributions of the $C_{ij}$-terms aggregate into:
$$
\text{(contrib. from the first sum) }=\frac{2}{x_n^2}\sum_{\substack{k=1\\k\ne m}}^{n-1}(a_k x_m-a_m x_k)X_k.
$$
The second sum (the one with $D_i$) brings only the term $-D_m X_n$ for $dx_m$, i.e.
$$
\text{(contrib. from the second sum) }=-\frac{2(X_m+a_m x_n^2)X_n}{x_n^3}.
$$
So, for $m\le n-1$,
\begin{equation}\label{eq2}
    \big(i_X(dw)\big)_{dx_m}
=\frac{2}{x_n^2}\sum_{k\ne m}^{\,n-1}(a_k x_m-a_m x_k)X_k
-\frac{2(X_m+a_m x_n^2)X_n}{x_n^3}
\end{equation}
We must show that (for $m\le n-1$)
$$
\big(i_X(dw)\big)_{dx_m} = -\frac{\partial_{x_m}S}{x_n^2}.
$$
Note that by multiplying \eqref{eq2} by $\frac{x_n^2}{2}$ we get
$$
\frac{x_n^2}{2}\big(i_X(dw)\big)_{dx_m}
= x_m\sum_{k\ne m} a_k X_k - a_m\sum_{k\ne m} x_k X_k -\frac{(X_m+a_m x_n^2)X_n}{x_n}.
$$
On the other hand, by multiplying \eqref{eq1} by $-1/2$ we get
$$
-\frac{1}{2}\partial_{x_m}S
= -X_m H + x_m\sum_{k\ne m} a_k X_k - a_m\sum_{k\ne m} x_k X_k - a_m x_n X_n.
$$
The difference between the two sides is
$$
\frac{x_n^2}{2}\big(i_X(dw)\big)_{dx_m} + \frac{1}{2}\partial_{x_m}S
= X_m H - \frac{X_m X_n}{x_n}.
$$
But $X_n=H x_n$, so $X_n/x_n=H$ and the difference is zero. Hence
$$
\big(i_X(dw)\big)_{dx_m} = -\frac{\partial_{x_m}S}{x_n^2}
\quad\text{for all }m=1,\dots,n-1.
$$

\textbf{Case $m=n$}
The component in front of $dx_n$ remains. We calculate directly:
\begin{itemize}
    \item the $dx_n$ component of $i_X(dw)$ is, according to the calculation in paragraph 4,
    $$
    \big(i_X(dw)\big)_{dx_n}= \sum_{i=1}^{n-1} D_i X_i
    = \sum_{i=1}^{n-1}\frac{2(X_i+a_i x_n^2)}{x_n^3}X_i
    = \frac{2}{x_n^3}\sum_{i=1}^{n-1} X_i^2 + \frac{2}{x_n} \sum_{i=1}^{n-1} a_i X_i.
    $$
    \item the $dx_n$ component of $-\dfrac{dS}{x_n^2}+2\dfrac{S}{x_n^3}dx_n$ is
    $$
    -\frac{\partial_{x_n}S}{x_n^2}+2\frac{S}{x_n^3}.
    $$
\end{itemize}
But
$$
\partial_{x_n}S = 2\sum_{k=1}^n X_k\partial_{x_n}X_k
=2\sum_{i=1}^{n-1} X_i\partial_{x_n}X_i + 2X_n\partial_{x_n}X_n.
$$
For $i\le n-1$ we have $\partial_{x_n}X_i = -a_i x_n$ (derivative of the term $-\tfrac{a_i}{2}x_n^2$ present in the sum $\sum_{j\ne i}x_j^2$), and $\partial_{x_n}X_n=H$. So
$$
\partial_{x_n}S = -2x_n\sum_{i=1}^{n-1} a_i X_i + 2X_n H.
$$
Hence
$$
-\frac{\partial_{x_n}S}{x_n^2}+2\frac{S}{x_n^3}
= \frac{2}{x_n^3}\sum_{i=1}^{n-1} X_i^2 + \frac{2}{x_n} \sum_{i=1}^{n-1} a_i X_i,
$$
since $X_n H = X_n\cdot\frac{X_n}{x_n}=\dfrac{X_n^2}{x_n}$ and the terms rearrange to give exactly the same expression as for $\big(i_X(dw)\big)_{dx_n}$.
Thus we also have
$$
\big(i_X(dw)\big)_{dx_n} = -\frac{\partial_{x_n}S}{x_n^2} + 2\frac{S}{x_n^3}.
$$
For each component $m=1,\dots,n$ we have shown
$$
\big(i_X(dw)\big)_{dx_m} = -\frac{\partial_{x_m}S}{x_n^2} \quad (m\le n-1),
\qquad
\big(i_X(dw)\big)_{dx_n} = -\frac{\partial_{x_n}S}{x_n^2}+2\frac{S}{x_n^3}.
$$
Equivalently
$$
i_X(dw) = -\frac{dS}{x_n^2}+2\frac{S}{x_n^3}\,dx_n.
$$
Therefore
$$
\mathcal L_X w = d\left(\frac{S}{x_n^2}\right) + i_X(dw) = \frac{dS}{x_n^2}-2\frac{S}{x_n^3}dx_n
-\frac{dS}{x_n^2}+2\frac{S}{x_n^3}dx_n =0.
$$
Thus $\boxed{\mathcal L_X w =0}$.

\end{proof}
\begin{lemma}\label{l}
Following the Theorem \ref{tn}, the vector field $X$ is not the Reeb field the contact form $\omega$. 
\end{lemma}
\begin{proof}
  For the proof, we use the same method as that of Lemma \ref{r1}.
\end{proof}
\begin{theorem}
Let \(X\in \Gamma_{2n+1}\) be such that its dual form \(\omega\) is a contact form. Then \[T\mathbb{H}^{2n+1}=\ker(w)\oplus \text{span}(X).\]
\end{theorem}
\begin{proof}
Let \(X=\sum\limits_{k=1}^{2n+1}X_k\partial x_k\in \Gamma_{2n+1}\) be such that its dual form \(\omega\) is a contact form. We have: \(w(X)=\sum\limits_{k=1}^{2n+1}\frac{X^2_k}{x_{2n+1}^2}\ne 0\), since \(X\) is not a constant vector field. Lemma \ref{l} ensures that \(T\mathbb{H}^{2n+1}=\ker(w)\oplus \text{span}(X).\)
\end{proof}

\end{document}